\documentclass[12pt,verbatim]{amsart}
\usepackage[T1]{fontenc}
\usepackage{amsfonts}
\usepackage{amsthm}
\usepackage{amsmath}
\usepackage{amssymb}
\usepackage{amscd}
\usepackage{graphicx}
\usepackage{caption}

\newcounter{minutes}
\divide\time by 60
\newcounter{hours}
\multiply\time by 60 \addtocounter{minutes}{-\time}

\theoremstyle{plain}
\newtheorem{thm}[equation]{Theorem}

\newtheorem{lem}[equation]{Lemma}
\newtheorem{example}[equation]{Example}
\newtheorem{prop}[equation]{Proposition}

\theoremstyle{definition}

\theoremstyle{remark}
\newtheorem{rem}[equation]{Remark}

\newcommand{\R}{{\mathbb R}}
\newcommand{\N}{{\mathbb N}}

\newcommand{\B}{\mathbf B}
\renewcommand{\S}{\mathbf S}
\newcommand{\U}{\mathbf U}

\newcommand{\eps}{\varepsilon}

\def\be{\begin{equation}}
\def\ee{\end{equation}}
\newcommand{\X}{\mathrm{X}}

\def\tang{\mathrel{\hbox{\rlap{%
\kern 0.2ex\hbox to 1.8ex{\hrulefill}}\raise2.6pt\hbox{$\bigcirc$}}}}

\numberwithin{equation}{section}

\begin{document}

\def\thefootnote{}
\footnotetext{ \texttt{\tiny File:~\jobname .tex,
          printed: \number\year-\number\month-\number\day,
          \thehours.\ifnum\theminutes<10{0}\fi\theminutes}
} \makeatletter\def\thefootnote{\@arabic\c@footnote}\makeatother

\title{On smoothness of quasihyperbolic balls}

\author{Riku Kl\'en}
\address{Department of Mathematics and Statistics, University of Turku, FI-20014 Turku, Finland}
\email{riku.klen@utu.fi}

\author{Antti Rasila}
\address{Department of Mathematics and Systems Analysis, Aalto University, P.O. Box 11100, FI-00076 Aalto, Finland}
\email{antti.rasila@iki.fi}

\author{Jarno Talponen}
\address{University of Eastern Finland, Institute of Mathematics, Box 111, FI-80101 Joensuu, Finland}
\email{talponen@iki.fi}

\begin{abstract}
We investigate properties of quasihyperbolic balls and geodesics in Euclidean and Banach spaces. Our 
main result is that in uniformly smooth Banach spaces a quasihyperbolic ball of a convex domain is $C^1$-smooth. The question about the smoothness of quasihyperbolic balls is old, originating back to the discussions of F.W. Gehring and M. Vuorinen in 1970's. To our belief, the result is new also in the Euclidean setting. We also address some other issues involving the smoothness of quasihyperbolic balls.

We introduce an interesting application of quasihyperbolic metrics to renormings of Banach spaces. 
To provide a useful tool for this approach we turn our attention to the variational stability of quasihyperbolic geodesics. Several examples and illustrations are provided.
\end{abstract}

\maketitle




\section{Introduction}


The \emph{quasihyperbolic metric} in $\R^n$ is a natural generalization of the hyperbolic metric, introduced  by Gehring \cite{GehringOsgood79,GehringPalka76} and his students Palka and Osgood in 1970's as a tool for studying quasiconformal mappings. Since its introduction, this metric has found numerous applications to the geometric function theory \cite{Heinonen01,Vuorinen88}. Furthermore, quasihyperbolic metric can be studied in more general settings than $\R^n$, such as Banach spaces and even in general metric spaces. It has particular significance in the infinite dimensional settings, where many traditional tools, such as the conformal modulus, cannot be used. This approach to geometric function theory was developed by J. V\"ais\"al\"a in a series of papers. This theory  is called ``the free quasiworld'' according to V\"ais\"al\"a (see \cite{Vai99}, and references therein). Introductory discussion and motivations for the study of this topic are presented in detail in the survey article \cite{KRT}. 

In this paper, we continue our investigation of the properties of this metric in Euclidean and Banach spaces, see \cite{Klen08,Klen09, RasilaTalponen12,RasilaTalponen14}. We will settle a long-standing problem of whether quasihyperbolic balls of a convex domain on a uniformly smooth Banach space are smooth (see Remark \ref{gehring_rugby}). Our smoothness considerations yield as a byproduct a new renorming technique of Banach spaces as well. In fact, it turns out that the quasihyperbolic metric is differentiable in a large dense set of points if the Banach space in question satisfies some modest regularity assumptions, e.g. it is separable or has a separable dual, see Theorem \ref{thm: asplund}. This approach is made possible by the fact that the underlying Banach space geometry is conveyed to the geometric properties of the quasihyperbolic metric. Thus we have access to the functional analysis machinery in studying the properties of the quasihyperbolic metric. This includes topics such as the convexity of balls and smoothness of geodesics, and an approach to analyzing these matters was developed in our earlier papers \cite{RasilaTalponen12,RasilaTalponen14}.

This paper is organized as follows. Shortly, in Section \ref{sect: motiv}, we explain the 
justification for the quasihyperbolic metric. After that we provide the main references and 
some more required definitions. In Section \ref{sect: smooth} we prove our main result involving the 
$C^1$-smoothness of the quasihyperbolic metric. As mentioned above, we also prove that under rather weak 
assumptions on the Banach space the metric is differentiable in a dense large set of points. Also, some 
results related to smoothness of the inner metrics are given in a purely metric setting.
In Section \ref{sect: ultra} we develop geometric tools by using non-standard analysis, stating roughly that 
if two families of paths are close after a quasihyperbolical state change, then they are close in absolute terms.  
Then, in Section \ref{sect: renorm}, we discuss a promising and unexpected connection between 
geometric function theory and functional analysis, namely using the quasihyperbolic metric to construct 
Banach space renormings. We apply the abovementioned tool in the proof of Theorem \ref{thm: unifconv}
and in this connection it becomes a kind of variational principle. The main goal there is approximating arbitrary equivalent norms with well behaved ones induced by quasihyperbolic balls. Finally, we give some counterexamples in Hilbertian and Euclidean spaces.

\subsection{Geometric motivation of the quasihyperbolic metric}\label{sect: motiv}

Suppose that $\X$ is a Banach space and $\Omega \subset \X$ is a domain with non-empty boundary. Denote by 
$d(x,\partial\Omega)$ the distance of the point $x\in \Omega$ from the boundary of $\Omega$. Then the {\it quasihyperbolic (QH) metric} on the domain $\Omega$ is defined by the formula
\begin{equation}
\label{qhdef}
k_\Omega(x,y) = \inf_{\gamma}\int_\gamma \frac{\|dz\|}{d(z,\partial\Omega)},
\end{equation}
where the infimum is taken over all rectifiable curves $\gamma$ in $\Omega$ connecting the points $x,y$ in $\Omega$. If the infimum is attained for some rectifiable curve $\gamma$, this curve is called a {\it quasihyperbolic geodesic}. If there is no danger of confusion, we write $k(x,y)$ instead of $k_\Omega(x,y)$.

The formula \eqref{qhdef} has several important special cases. For $\X=\R^n$, $n\ge2$ and $\Omega=\mathbb{H}^n=\{ x\in \R^n:x_n>0\}$, i.e., the upper half-space, we obtain the well-known {\it hyperbolic metric} $\rho$ (see \cite{Beardon}), also known as the {\it Poincar\'e metric}. The hyperbolic metric can also be developed in the unit disk $\mathbb{B}^2$ by using the formula
\[
\rho_{\mathbb{B}^2}(x,y)=\inf_\gamma \int_{\gamma}  \frac{2|dz|}{1-|z|^2}\quad x,y\in \mathbb{B}^2,
\]
where the infimum is taken over all regular curves $\gamma$ connecting $x$ and $y$. This metric is conformally invariant in the following sense. Suppose that $w=f(z)$ is a conformal mapping of the unit disk onto itself. Then, by Pick's lemma, we have the identity
\[
\bigg|\frac{dw}{dz}\bigg| = \frac{1-|w|^2}{1-|z|^2},
\text{ or }
\frac{|dw|}{1-|w|^2} = \frac{|dz|}{1-|z|^2}.
\]
This means that, for any regular curve $\gamma$ in the unit disk, we have
\[
\int_{f\circ\gamma} \frac{|dw|}{1-|w|^2} = \int_{\gamma}  \frac{|dz|}{1-|z|^2}.
\]
Note that the denominators above are asymptotically equivalent to that in \eqref{qhdef}.

By using conformal invariance, the hyperbolic metric can be studied for other simply connected domains in the plane, and in the case of the half-plane it coincides with the metric defined by \eqref{qhdef}. But even in $\R^n$, the hyperbolic metric cannot be defined for domains other than half-spaces and balls for $n\ge 3$. The method based on conformal invariance does not apply to
general Banach spaces.

In general, the quasihyperbolic metric is not conformally invariant, but it behaves well under conformal and even quasiconformal mappings. This fact is of particular importance in infinite dimensional spaces, where many convenient tools for studying quasiconformal mappings, such as local compactness and measure, are not available. Besides quasiconformal mappings, the quasihyperbolic metric has recently found novel and interesting applications in other fields of geometric analysis as well. For example, quasihyperbolic metric has been recently used in study of the Poincar\'e inequality \cite{jiang-kauranen,koskela-onninen-tyson}.

From basic analysis of the hyperbolic disk and conformal invariance, it immediately follows that the hyperbolic geodesics originating from the point $x\in \Omega$ are always orthogonal to the surfaces of hyperbolic balls centered at $x$ (see Figure \ref{fig:hyper}). It is not obvious that a similar property holds for the quasihyperbolic metric. This very useful connection between quasihyperbolic geodesics and balls will be established in Theorem \ref{orto}, and then exploited to obtain further results on quasihyperbolic balls.
We refer to the book of Vuorinen \cite{Vuorinen88} for the basic properties of the quasihyperbolic metric in $\R^n$, and the comprehensive survey article of V\"ais\"al\"a \cite{Vai99} for the basic results in Banach spaces.

  \begin{figure}[ht]
    \includegraphics[width=6.5cm]{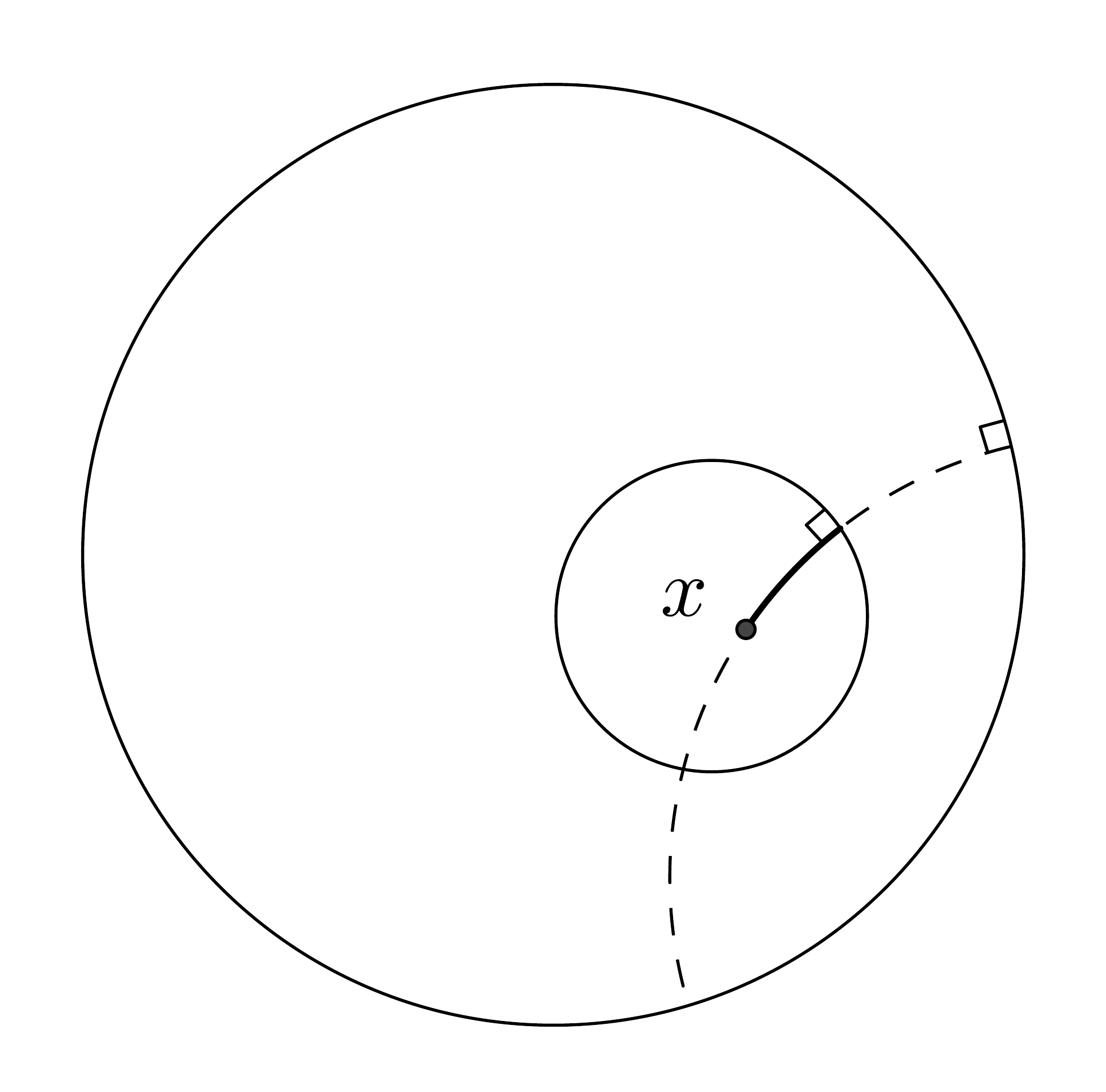}
    \caption{A geodesic radius in the hyperbolic disk.} \label{fig:hyper}
  \end{figure}



\subsection{Preliminaries}

We refer to \cite{Beardon} , \cite{DGZ} , \cite{FA_book}, \cite{Heinonen01}, \cite{Vuorinen88} and \cite{Vai99} for background information.
Many of the arguments here are written concisely, using reasoning similar to that in \cite{RasilaTalponen12} and \cite{RasilaTalponen14}.
By a \emph{domain} $\Omega \subset \X$ of a Banach space we mean an open path-connected subset with a non-empty boundary $\partial \Omega$. We denote by $\B(x_0 ,r)$ closed balls 
(resp. by $\U(x_0 ,r)$ open balls) with center $x_0$ and 
radius $r$ in a metric space and by $\S(x_0 ,r)=\partial \B(x_0 ,r)$ the corresponding spheres. The 
symmetric Hausdorff distance between subsets of a metric space is denoted by $d_H$.

We assume throughout that all Banach spaces considered have the Radon-Nikodym Property (RNP) 
which can be formulated as follows: Each Lipschitz path $\gamma \colon [0,1] \to \X$ is a.e. differentiable 
and the path can be recovered by Bochner integrating its derivative,
\[\gamma(t) = \gamma(0) + \int_{0}^t \gamma' (s)\ ds,\quad t\in [0,1],\]
see \cite{Diestel}.
Recall that the {\it uniform smoothness} of a Banach spaces $\X$ is defined by the following condition on $\|\cdot\|$ norm:
\[\mu_{(\X,\|\cdot\|)} (\tau)=\sup\{(\|y+h\|+\|y-h\|)/2 -1,\ y,h\in \X,\ \|y\|=1,\ \|h\|=\tau\},\]
\[
\lim_{\tau\to 0^{+}}\frac{\mu_{(\X,\|\cdot\|)}(\tau)}{\tau}=0.
\]

Suppose $\mathcal{U}$ is an {\it ultrafilter over} $\N$. If $(a_n)\subset \R$ and $a\in \R$ are such that
\[\forall\ \eps>0\ \{n\in \N\colon |a_n-a|<\eps\}\in \mathcal{U}\]
then this is by definition the statement
\[\lim_{\mathcal{U}} a_n =a.\]
We will frequently apply  the fact that 
$\liminf_{n\to\infty} a_n \leq \lim_{\mathcal{U}} a_n \leq \limsup_{n\to\infty} a_n$.

Next we recall the definition of an ultrapower $\X^\mathcal{U}$ of a Banach space $\X$.
First, consider the $\X$-valued $\ell^\infty$ space, $\ell^\infty (\X)$. 
This is obtained by replacing the real coordinates by vectors of $\X$ and the norm is defined by 
$\|(x_n )\|_{\ell^\infty (\X)}=\sup_n \|x_n \|_X$. This is a Banach space. 
Then 
\[N_\mathcal{U} :=\{(x_n) \in \ell^\infty (\X)\colon \lim_\mathcal{U} \|x_n\|_X =0\}\subset \ell^\infty (\X)\]
is a closed subspace. The ultrapower is the following quotient space
\[\X^\mathcal{U} = \ell^\infty (\X) / N_\mathcal{U}.\]
Observe that $(x_n ),(y_n ) \in \ell^\infty (\X)$ are representatives of an element $x\in \X^\mathcal{U}$ if and only if 
\[ \lim_\mathcal{U} \|x_n - y_n \|_\X =0.\]
See \cite{Heinrich} for more information on these constructions.

We will frequently use the fact that in bounded subdomains $D \subset \Omega$ uniformly separated form the boundary $\partial \Omega$ the quasihyperbolic metric is equivalent to the metric induced by the norm, 
see also the proof of Theorem \ref{thm: UF}.

\section{On smoothness of quasihyperbolic balls in infinite-dimensional setting}\label{sect: smooth}

Let us consider a convex domain $\Omega$ in a Banach space. It follows from the arguments provided by V\"ais\"al\"a and Martio in \cite{MartioVaisala} that there exists a quasihyperbolic geodesic between any two points 
in a convex domain of a locally weak-star compact (e.g. reflexive) space.

Let us consider paths as elements in $C([0,1],\X)$, the space of continuous functions $[0,1]\to \X$ with the $\sup$-norm, which is a Banach space. We will consider rectifiable paths with finite quasihyperbolic length parametrized by their quasihyperbolic length, i.e. having constant speed.

We will study the following type (multi)map: $\Lambda (x,y) \mapsto \{\gamma\}$ which assigns to each endpoints 
all the corresponding quasihyperbolic geodesics. 

\begin{prop}\label{prop: single}
In a reflexive, strictly convex Banach space with convex domain $\Omega$ the mapping $\Lambda$ is well-defined and single-valued, $\Omega \times \Omega \to C([0,1],\X)$. 
\end{prop}
\begin{proof}
Indeed, under the assumption the geodesics exist and are unique, see \cite{RasilaTalponen12}.
\end{proof}

\begin{lem}\label{lm: PLconvex}
The path length functional 
\[\ell_k \gamma =\int_\gamma \frac{\|\gamma' (t)\|}{d(\gamma(t),\partial\Omega)}\ dt\]
is convex. 
\end{lem}
\begin{proof}
Take the point-wise weighted average of two rectifiable paths $\gamma_0$ and $\gamma_1$, 
$\gamma_s (t) := (1-s) \gamma_0 (t) + s \gamma_1 (t)$, $s\in [0,1]$. Then 
\[\ell_k \gamma_s \leq  (1-s) \ell_k \gamma_0 + s \ell_k \gamma_1 .\]
This follows from the inequality 
\begin{equation}\label{eq: *}
\frac{\|((1-s)\gamma_0 + s\gamma_1 )' \|}{d((1-s)\gamma_0 + s\gamma_1 ,\partial\Omega)} \leq 
\frac{(1-s)\|\gamma_{0}' \|}{d(\gamma_0 ,\partial\Omega)} +\frac{ s\|\gamma_{1}'\| }{d(\gamma_1 ,\partial\Omega)},
\end{equation}
see \cite[(4.6)]{RasilaTalponen12}.
\end{proof}

Suppose next that $\X$ is an {\it Asplund space}. This is a weaker condition than reflexivity and is dual to the RNP condition of a Banach space.
The Asplund property has the following equivalent formulation: Each continuous convex function $f\colon \X \to \R$ is Fr\'echet differentiable
in a generic set, i.e. a dense $G_\delta$ subset.

 \begin{thm}\label{thm: asplund}
 Suppose that $\X$ is an Asplund space and $\Omega\subset \X$ is a convex domain. Then the quasihyperbolic metric $k(x ,y)$ is Fr\'echet smooth in a generic subset of $\Omega\times \Omega$. If we fix one coordinate, 
 $k_{x_0}(y):=k(x_0 , y)$ is Fr\'echet smooth in a generic subset of $\Omega$.  
\end{thm}
\begin{proof}
In the case with a strictly convex reflexive space the function 
\[k(x , y)=\ell_k \Lambda(x , y)\] 
is convex by the above Proposition \ref{prop: single} and Lemma \ref{lm: PLconvex}. 
Indeed, if $\gamma_0$ (resp. $\gamma_1$) is a geodesic connecting $x_0$ to $y_0$
(resp. $x_1$ to $y_1$), then 
\begin{multline*}
k((1-s)x_0 + sy_0 ,  (1-s)x_1 + sy_1 ) \leq \ell_k (\gamma_s ) \\
\leq (1-s)\ell_k (\gamma_0 ) + s \ell_k (\gamma_1 ) = (1-s)k(x_0, y_0 ) + s k(x_1 , y_1 ).
\end{multline*}
It is easy to see that $k$ is Lipschitz on domains uniformly separated from the boundary $\partial \Omega$. Thus Asplund property applies.

In the non-strictly convex, non-reflexive case we still have the convexity of $k$. The statement of the theorem can then be seen for example by approximation with quasigeodesics in place of $\Lambda(x , y)$. 

The latter part of the statement is seen in a similar fashion.
\end{proof}

\begin{rem}
The above statement also holds if Asplund condition is replaced by separability and Fr\'echet differentiability by G\^ateaux differentiability.
Indeed, it is know that separable Banach spaces are weak Asplund spaces (where the formulation runs analogously where G\^ateaux differentiability appears in place of Fr\'echet differentiability).
\end{rem}

\begin{thm}
If $\Omega$ is convex and the norm of $\X$ is uniformly convex, then the mapping $(\Lambda)' \colon \Omega^2 \to L^1 (\X)$,
$(x,y)\mapsto \gamma'$, is 
$\|\cdot\|_{\X\oplus\X}$-$\|\cdot\|_{L^1 (\X)}$-continuous.
\end{thm}
\begin{proof}[Sketch of proof]
Let $x_n \to x$, $y_n \to y$ in norm in $\Omega$. Then the convergences hold in the quasihyperbolic metric as well by the local bi-Lipschitz equivalence of the metrics.

According to the above observations there are unique geodesics $\gamma_n$ and $\gamma$ be geodesics connecting $x_n$ to $y_n$ and $x$ to $y$, 
respectively. Assume further that these are parameterized by constant quasihyperbolic path length growth. 
Analyze the quasihyperbolic  length of the averages $\frac{1}{2}(\gamma+\gamma_n)$. It turns out by using the modulus of convexity 
on \eqref{eq: *} that $\|\gamma' -\gamma_{n}' \|\to 0$ in measure as $n\to\infty$. Indeed, otherwise 
$\liminf_{n\to\infty} \ell_k (\frac{1}{2}(\gamma+\gamma_n))< \ell_k (\gamma)$ stating that 
\[
\liminf_{n\to\infty} k\Big(\frac{x+x_n}{2} ,\frac{y+ y_n}{2}\Big) < k(x,y),\]
which is impossible. Since these derivatives are essentially bounded according to the 
parametrization of the paths, we obtain 
\[
\int_{0}^1 \|(\gamma -\gamma_n )' (t) \|\ dt\to 0\text{ as }n\to\infty.
\] 
\end{proof}

\begin{thm}\label{thm: UF}
Let $\X$ be a uniformly smooth Banach space and let $\Omega\subset \X$ be a convex domain. Then the quasihyperbolic  balls are smooth in the sense that $k(x_0 , \cdot)$, $x_0 \in \Omega$, is continuously Fr\'echet differentiable away from $x_0$.
\end{thm}
\begin{proof}

Let $B$ be a quasihyperbolic  ball as above. Observe that there is a radius $\epsilon>0$ such that $\|\cdot\|$ and $k_0 (\cdot) := k(x_0 , \cdot)$ are bi-Lipschitz equivalent in 
\[(\partial B)^\epsilon = \{x\in \Omega\colon d(x,\partial B)<\epsilon\}.\]
This is due to the fact that 
on $B^\epsilon$ the weight $\frac{1}{d(x,\partial \Omega)}$ is bounded. 
This in turn is based on the observation that 
\[
\int_{0}^T \frac{dt}{L(t)} = \infty,
\]
for any $1$-Lipschitz mapping $L$ with $L(t)\to 0$ as $t\nearrow T$.
Here $T$ represents the norm length of a rectifiable path and $L$ represents $d(\gamma(t),\partial \Omega)$.

Let us verify that $k_0$ is indeed Fr\'echet differentiable away from the origin. Towards this, let 
$\gamma\colon [0,\ell]\to \Omega$ be a quasihyperbolic  geodesic joining two points $\gamma(0)=x_0 ,\gamma(\ell)\in B$
and parameterized by the \emph{norm} path length. By thinking of $B$ as a closure of incremental sequence of quasihyperbolic  balls 
it is easy to see that $\gamma$ is fully contained in $B$. 

We wish to show that
\begin{equation}\label{eq: sm0}
\frac{1}{\|h\|}(k(\gamma(0),\gamma(\ell)+h)+ k(\gamma(0),\gamma(\ell)-h)-2k(\gamma(0),\gamma(\ell))) \to 0,
\end{equation}
as $\|h\|\to 0$ uniformly, not depending on the particular endpoints or length $\ell$. However, we will assume that $\ell$ is uniformly bounded 
away from zero.

Fix $h$ such that $\gamma(\ell)\pm h \in \Omega$. Write $t=\sqrt{\|h\|}$.
Define new paths $\gamma_+$ and $\gamma_-$ as follows: on the segment $[0,\ell-t]$ they coincide with $\gamma$ and 
\[\gamma_\pm (\ell - t+s)=\gamma(\ell- t +s) \pm s^2 \frac{h}{\|h\|}\ \mathrm{for}\ 0\leq s \leq t.\]
Note that this definition is not sensible if $\ell$ is allowed to have values smaller than $t$..
The asymptotics of \eqref{eq: sm0} can be estimated by studying the sum of the lengths of $\gamma_\pm$. We are 
only required to study the behavior at parameter values on $[\ell-t,\ell]$.

Note that we may write $\mu_{\|\cdot\|}(\tau)=\tau \epsilon(\tau)$ where $\epsilon(\tau)\searrow 0$ as $\tau \to 0$. Write 
\[d^* = \sup_{y\in B^\epsilon} \frac{d}{dt}\frac{1}{t}\Big\vert_{t=d(y,\partial \Omega)}.\]

Recall that $\|\gamma' \|=1$ a.e., by the parameterization. Note that 
\[\|(\gamma_\pm (\ell -t + s) - \gamma (\ell -t + s))^\prime \|=\Big\|\pm\frac{d}{ds}s^2 \frac{h}{\|h\|}\Big\|=2s .\] 
Thus by using the definition of the modulus of smoothness we obtain that
\[\sum_\pm \|\gamma_\pm ^{\prime} (\ell -t + s)\| \leq 2 (1+\mu_{\|\cdot\|} (2s)) .\]
Also note that 
\[\frac{1}{d(\gamma_\pm (\ell -t + s), \partial \Omega)}\leq \frac{1}{d(\gamma (\ell -t + s), \partial \Omega)} + s^2 d^*\]
by the mean value principle. We obtain that
\begin{multline*}
k(\gamma(0),\gamma(\ell)+h)+ k(\gamma(0),\gamma(\ell)-h)-2k(\gamma(0),\gamma(\ell))\\
\leq \ell_{k}(\gamma_+)+\ell_{k}(\gamma_- ) - 2 \ell_{k}(\gamma)\\
= \int_{s=0}^{t} \frac{\|\gamma_+ ^{\prime} (\ell -t + s)\|}{d(\gamma_+ (\ell -t + s), \partial \Omega)}\ ds + \int_{s=0}^{t} \frac{\|\gamma_- ^{\prime} (\ell -t + s)\|}{d(\gamma_- (\ell -t + s), \partial \Omega)}\ ds - 2 \int_{0}^{t} \frac{ds}{d(\gamma (\ell -t + s), \partial \Omega)}\\
\leq \int_{s=0}^{t} 2 (1+\mu_{\|\cdot\|} (2s)) \left(\frac{1}{d(\gamma (\ell -t + s), \partial \Omega)} + s^2 d^*\right)ds - 2\int_{0}^{t} \frac{ds}{d(\gamma (\ell -t + s), \partial \Omega)}.
\end{multline*}
We claim that the bottom value converges to $0$ faster than $t^2 \to 0$.  
Since $s^2 \leq t^2 \to 0$ we are only required to investigate the term involving $\mu_{\|\cdot\|}$.

Note that since $y\mapsto \frac{1}{d(y,\partial \Omega)}$ is Lipschitz on $B$ (see the proof of Theorem \ref{thm: UF}), there is small enough $T$ (or $\|h\|$) such that 
\[1/d(\gamma (\ell -t + s), \partial \Omega) \leq 2 /d(\gamma (\ell), \partial \Omega)\ \mathrm{for}\ t\leq T .\]
Let us  estimate
\begin{multline*}
\frac{1}{t^2} \int_{0}^{t} 2\mu_{\|\cdot\|} (2s)\left(\frac{1}{d(\gamma (\ell -t + s), \partial \Omega)} + s^2 d^*\right) ds\\
\leq \frac{1}{t^2} \int_{0}^{t}4s \epsilon(2s)\left(\frac{2}{d(\gamma (\ell), \partial \Omega)} + s^2 d^*\right) ds\\
\leq \frac{1}{t^2} \Bigg\vert_{s=0}^{t} \epsilon(2t)\left(\frac{4s^2}{d(\gamma (\ell), \partial \Omega)} + s^4 d^*\right)\\
=\epsilon(2t)\left(\frac{4}{d(\gamma (\ell), \partial \Omega)}+t^2 d^*\right)\to 0,\quad t\to 0.
\end{multline*}

In fact, we observe by going through the calculations that if $p,r\in [1,2)$,  $r\leq (p+1)/2$, $r<3/2$  and 
$\mu_{\|\cdot\|} (\tau) = \tau^p \epsilon(\tau)$ where $\epsilon(\tau) \searrow 0$ as $\tau\to 0$, then we have 
\[\frac{1}{\|h\|^r}(k(x_0 ,x+h)+ k(x_0 ,x-h)-2k(x_0 ,x)) \to 0\]
as $\|h\|\to 0$ uniformly in any annulus whose distance to $x_0$ and $\partial \Omega$ is strictly positive.
In particular, for $p=r=1$ this states that uniform smoothness of the norm implies Fr\'echet differentiability 
of the function $k_0 (\cdot)$ in open annuli $(\partial B)^\eps$ as above, since the convergence is uniform in $\|h\|$. 
In fact, since the convergence is uniform both in $x$ and $\|h\|$ we have uniform Fr\'echet differentiability, see \cite[pp. 242, 289]{FA_book}.
This means that taking the Fr\'echet derivative $x \mapsto (k_0)' (x)$ is a uniformly continuous map 
$U \to \X^*$ for any open convex $U\subset \Omega$ uniformly separated from the $x_0$ and $\partial \Omega$.
\end{proof}

\begin{rem}
\label{gehring_rugby}
The question of smoothness of quasihyperbolic balls has been studied at least since 1979, when this question was posed by M. Vuorinen to F.W. Gehring. In an unpublished note \cite{gehringnote}, Gehring proposed that quasihyperbolic balls $\B_k(0,r)$, $r>0$ in the two-dimensional domain $D=\{z\in \R^3 : -1< z_2	 <1\}$ are not smooth. The above result shows that this is not a valid counterexample. See also \cite{Vuorinen2007}.
\end{rem}

The following result can be viewed as a kind of metric orthogonality between the quasihyperbolic  spheres and the 
quasihyperbolic  geodesics emanated from the quasihyperbolic  center of these spheres.

\begin{thm}
\label{orto}
Let $\Omega$ be a domain as above and $x_0 , x \in \Omega$. Suppose that $\gamma\colon [0,1]\to \Omega$ is a quasihyperbolic geodesic and 
$r=k(x_0 , x),\ \gamma(0)=x_0 ,\ \gamma(1)=x$. Then
\[\lim_{t\to 1^- } \frac{d_{\|\cdot\|} (\gamma(t),\S_{k}(x_0 , r))}{\|\gamma(t) - \gamma(1)\|}=1.\]
In the above limit one may also replace $\|x-y\|$ by $k(x,y)$. 
\end{thm}

The above statement holds similarly in a path-connected metric space with a continuous weight.
Surprisingly, the above orthogonality is mainly setting-free, it only depends on the conformality of the 
inner metric.  
 
\begin{proof}
We will prove the case with a general strictly positive continuous weight function $w$ in place of $1/d(x,\partial \Omega)$.  

We first observe that 
\[\limsup_{t\to 1^- } \frac{d_{\|\cdot\|} (\gamma(t),\S_{w}(x_0 , r))}{\|\gamma(t) - \gamma(1)\|}\leq 1.\]
Thus it suffices to check that 
\[\liminf_{t\to 1^- } \frac{d_{\|\cdot\|} (\gamma(t),\S_{w}(x_0 , r))}{\|\gamma(t) - \gamma(1)\|}\geq 1.\]

Let 
\[\underline{w}_t = \inf\{w(x)\colon \|x-\gamma(1)\|\leq 3\|\gamma(t)-\gamma(1)\|\}\]
and
\[\overline{w}_t = \sup\{w(x)\colon \|x-\gamma(1)\|\leq 3\|\gamma(t)-\gamma(1)\|\}.\]
Note that 
\[\lim_{t\to 1^-} \underline{w}_t=\lim_{t\to 1^-} \overline{w}_t = w(\gamma(1)).\]

Choose $t < 1$ such that 
\[\frac{\underline{w}_s}{\overline{w}_s}>\frac{1}{2},\quad t  < s < 1.\]
Suppose that $\lambda$ is a path starting from $\gamma(t)$. Suppose first that there exists $\lambda$ such that 
\begin{equation}\label{eq: set}
D=\big\{x\in \Omega\colon \|x-\gamma(1)\|\leq 3\|\gamma(t)-\gamma(1)\|\big\}.
\end{equation}
This means that 
\[\ell_w (\lambda) \geq 2 \underline{w}_{t}  \|\gamma(t)-\gamma(1)\|> \overline{w}_{t}  \|\gamma(t)-\gamma(1)\|\geq d_w (\gamma(t),\S_{w}(x_0 , r)).\]
It follows that $\ell_w (\lambda)> \ell_w (\gamma|_{[t , 1]})$. Therefore in studying $w$-short paths $\lambda$ between $\gamma(t)$ 
and the set $\S_{w}(x_0 , r)$ we may restrict to the case where the paths are included in the above set $D$ in \eqref{eq: set}.

Fix $y_{t}\in \S_{w}(x_0 , r) \cap D$. Suppose without loss of generality that $\lambda$ is a $w$-geodesic between $\gamma(t)$ and $y_{t}$. 
Indeed, if there does not exist a geodesic between the given points then we may pass on to  an approximating sequence of  paths $\lambda_n$ 
and the rest of the argument does not change considerably. 

Note that 
\[\ell_{\|\cdot\|}(\gamma|_{[t,1]}) \underline{w}_{t} \leq  \ell_{\|\cdot\|}(\lambda)  \overline{w}_{t}.\]
Thus 
\[\frac{ \underline{w}_{t}}{\overline{w}_{t}}\leq \frac{\ell_{\|\cdot\|}(\lambda)}{\ell_{\|\cdot\|}(\gamma|_{[t ,1]})}\]
where the left hand side tends to $1$ as $t \to 1$.

To estimate the right hand side we observe similarly as above that 
\[\|\gamma(t) - \gamma(1))\| \leq \ell_{\|\cdot\|}(\gamma|_{[t , 1]})\leq \frac{\overline{w}_{t}}{ \underline{w}_{t}}\|\gamma(t) - \gamma(1))\|\]
and an analogous statement holds for $\lambda$ also.
Thus we obtain
\[1\leq \liminf_{t \to 1^- } \frac{\ell_{\|\cdot\|}(\lambda)}{\ell_{\|\cdot\|}(\gamma|_{[t ,1]})}=\liminf_{t \to 1^-} \frac{\|\gamma(t) - y_{t}\|}{\|\gamma(t) - \gamma(1)\|}.\]
This yields the claim since the selections of the points $y_t$ were arbitrary.
\end{proof}

The following result says that a quasihyperbolic ball does not have cusps.

\begin{thm}
For a proper subdomain $\Omega$ of $\X$, radius $r>0$ and $y\in \partial \B_k(x,r)$, 
let $\gamma$ be a quasihyperbolic geodesic joining $x$ and $y$. For $z\in \gamma$ we have
\[
  \B_{\|\cdot\|}\left(z,\frac{|z-y|}{1+u}\right) \subset \B_{k}(x,r),
\]
where $u=|z-y|/d(z,\partial \Omega)$.
\end{thm}

The proof follows verbatim the proof of Theorem 2.8 in \cite{Klen08} and is reproduced here for the sake of convenience.

\begin{proof}
By the choice of $z$ we have
\[
  r = k(x,y) = k(x,z)+k(z,y)
\]
and by the triangle inequality for $w\in \B_k\big(z,k(z,y)\big)$ we have
\[
    k(x,w)\le k(x,z)+k(z,w) < r.
\]
Now
\[
  \B_k \big(z,k (z,y)\big) \subset \B_k (x,r).
\]
  
Next we will apply some known estimates involving the quasihyperbolic metric. These have been proven in the Euclidean setting, 
see \cite{Vuorinen88}, but since the proofs do not depend on the particular choice of the norm, we may use them in $\Omega$ as well.
  
We have
\[
    \B_{\|\cdot\|} \left( z, \left( 1-e^{-k(z,y)} \right) d(z,\partial \Omega) \right) \subset \B_k \big(z,k(z,y)\big),
\]
and thus
\[
  \B_{\|\cdot\|} \left( z, \left( 1-e^{-k(z,y)} \right) d(z,\partial \Omega) \right) \subset \B_k (x,r).
\]
It holds that 
$k(z,y)\ge \log( 1+|z-y|/d(z,\partial \Omega))$, and therefore
\begin{eqnarray*}
  \left( 1-e^{-k(z,y)} \right) d(z,\partial \Omega) & \ge & \left(1-\frac{d(z,\partial \Omega)}{d(z,\partial \Omega)+|z-y|}\right)d(z,\partial \Omega)\\
  & = & \frac{|z-y|}{1+u}
\end{eqnarray*}
for $u=|z-y|/d(z,\partial \Omega)$. Now
  \[
    \B_{\|\cdot\|}\left(z,\frac{|z-y|}{1+u}\right) \subset \B_{\|\cdot\|} \left( z, \left( 1-e^{-k(z,y)} \right) d(z,\partial \Omega) \right),
  \]
  and the claim follows.
\end{proof}

\section{Interlude: Variational principle via ultrapowers}\label{sect: ultra}
We will subsequently apply the following tool which can be seen as a kind of variational principle.

\begin{thm}\label{thm: var}
Let $\X$ be a superreflexive Banach space, let $w\colon \X  \to [a,\infty)$, $a>0$, be a weight function which is Lipschitz continuous 
and let $\gamma_n , \lambda_n \colon [0,1] \to \X$, $\gamma_n (0) = \lambda_n (0) =x_0\in \X$, be sequences of Lipschitz continuous paths such that 
\[\|\gamma_{n}' \| ,\ \|\lambda_{n}'\| \in [1/C , C]\]
holds a.e. and 
\begin{equation}\label{eq: unif}
\left\|\frac{\gamma_n (t+h) - \gamma_n (t-h)}{h} - 2\gamma_{n}' (t)\right\|, \left\|\frac{\lambda_n (t+h) - \lambda_n (t-h)}{h} - 2\lambda_{n}' (t)\right\|
\leq \beta(h),
\end{equation}
whenever sensible for a.e. $t\in [0,1]$, where $\beta(h)\searrow 0$ as $h\searrow 0$.
If 
\[\int_{0}^{1} \|\gamma_{n}' (t) w(\gamma_n (t)) - \lambda_{n}' (t) w(\lambda_n (t))\|\ dt \to 0,\quad n\to\infty\]
then 
\[\|\gamma_{n} (t) - \lambda_{n} (t) \|\to 0,\quad n\to\infty,\ t\in [0,1].\]
\end{thm}
 
\begin{lem}\label{lm: exact}
Let $\X$ be a Banach space with the RNP, $w\colon \X  \to (0,\infty)$ a weight function which is Lipschitz continuous 
and let $\gamma, \lambda \colon [0,1] \to \X$, $\gamma (0) = \lambda (0) =x_0$,  be Lipschitz paths such that for some $C>0$
\[\|\gamma' \| ,\ \|\lambda'\|  \in [1/C , C]\]
holds a.e. If 
\[\gamma' (t) w(\gamma (t)) = \lambda' (t) w(\lambda (t))\]
a.e. then $\gamma = \lambda$. 
\end{lem}
\begin{proof}
We immediately observe that $\gamma'(t)$ and $\lambda' (t)$ are linearly dependent for a.e. $t$.

If the paths coincide in an initial segment $[0,s] \subset [0,1]$, we may disregard it and start both the paths in the largest coordinate 
$s$ such that the above coincidence holds. Now, assume to the contrary that $s<1$ and rename $\gamma (t) = \gamma_{|[s,1]}(s+t(1-s))$,
$\lambda (t) = \lambda_{|[s,1]}(s+t(1-s))$.

We denote by $\nu_w$ the modulus of continuity of $w$, so by the assumptions $\nu_w (t) \leq a t$ for some $a>0$.
Observe that 
\[\|\gamma' (t) - \lambda' (t)\|\leq 2C\left(\frac{w(x_0 )+ \nu_w (\|\gamma(t)-\lambda(t)\|)}{w(x_0 )- \nu_w (\|\gamma(t)-\lambda(t)\|)}-1\right)\]
for small values of $t$, and, asymptotically, the right-hand side is dominated by 
\[\frac{8C\nu_w (\|\gamma(t)-\lambda(t)\|) }{w(x_0 )}.\]

Taking into account 
\[\|\gamma (t) - \lambda (t)\| \leq \int_{0}^t \|\gamma' (s) - \lambda' (s)\|\ ds,\] 
and applying change of variable
\[
u(t)= \int_{0}^t \|\gamma' (s) - \lambda' (s)\|\ ds,\quad 
\frac{du}{dt}=\|\gamma' (t) - \lambda' (t)\|,
\] we get 
\[\frac{du}{dt} \leq \frac{8C\nu_w (u(t)) }{w(x_0 )},\]
\[\int_{0}^{u(T)} \frac{w(x_0 )\ du}{8C\nu_w (u)} \leq \int_{0}^T 1\ dt.\]

However, if $u(T)>0$ the left hand side inessential integral diverges, being  an integral of a function 
bounded below by $b/u$ where $b>0$ is a suitable constant. This case is of course impossible and contradicts the choice of the parameter $s$.
\end{proof}

\begin{proof}[Proof of Theorem \ref{thm: var}]
We assume all the conditions in the statement of the theorem.

In order to verify that 
\[(\gamma_{n} (t) - \lambda_{n} (t)) \to 0,\quad n\to\infty\]
for all $t$ it suffices to check that for each subsequence $(n_k)\subset \N$ one can find a further subsequence 
$(n_{k_j})$ such that 
\begin{equation}\label{eq: nkj}
(\gamma_{n_{k_j}} (t) - \lambda_{n_{k_j}} (t)) \to 0,\quad j\to\infty,\quad t\in[0,1],
\end{equation}
or equivalently 
\[\liminf_{k\to\infty} \|\gamma_{n_{k}} (t) - \lambda_{n_{k}} (t))\|=0,\quad t\in [0,1].\]

Fix $(n_k)\subset \N$. Now, observe that 
\[\int_{0}^{1} \|\gamma_{n_k}' (t) w(\gamma_{n_k} (t)) - \lambda_{n_k}' (t) w(\lambda_{n_k} (t))\|\ dt \to 0,\quad n\to\infty\]
implies that there is a subsequence $(n_{k_j})$ such that 
\begin{equation}\label{eq: wto}
\|\gamma_{n_{k_j}}' (t) w(\gamma_{n_{k_j}} (t)) - \lambda_{n_{k_j}}' (t) w(\lambda_{n_{k_j}} (t))\|\to 0
\end{equation}
for a.e. $t$. 

Let $\mathcal{U}$ be a non-principal ultrafilter over $\N$ such that
\begin{equation}\label{eq: nkjU} 
\{n_{k_j}\colon j\in\N\}\in \mathcal{U}.
\end{equation}
Then, to verify \eqref{eq: nkj} for some subsequence of $(n_{k_j})$ (which is sufficient), it is enough to show that 
\[\lim_{\mathcal{U}} \|\gamma_{n} (t) - \lambda_{n} (t)\| =0,\quad t\in [0,1].\]

Let $\X^\mathcal{U}$ be the ultrapower of $\X$. 
It is well known that the superreflexivity of $\X$ implies the reflexivity of $\X^\mathcal{U}$.
Note that in particular $\X^\mathcal{U}$ has the RNP.

If $(x_n) \in \ell^\infty (\X)$ is a representative of $x\in \X^\mathcal{U}$ we denote $x=(x_n)^\mathcal{U}$.
We define a weight function $w^\mathcal{U}$ on $\X^\mathcal{U}$ by 
\[w^\mathcal{U}((x_n)^\mathcal{U})=\lim_{\mathcal{U}} w(x_n).\]
This definition is proper due to uniform continuity of $w$.
It follows readily from the construction of this weight that it satisfies the same hypothesis as $w$ in the assumptions.

Let 
\[\gamma(t):=(\gamma_n (t))^\mathcal{U} , \lambda (t) := (\lambda_n (t))^\mathcal{U} \colon [0,1] \to \X^\mathcal{U}.\]
It follows readily that these paths are $C$-Lipschitz and start from the point $x_{0}^\mathcal{U} :=(x_0 , x_0 ,x_0 ,\ldots)^\mathcal{U}$.
Since $\X^\mathcal{U}$ has the RNP and $\gamma$ and $\lambda$ are Lipschitz we observe that 
these paths can be recovered by integrating their derivatives,
\[\gamma(t)=x_{0}^\mathcal{U} + \int_{0}^t \gamma' (s)\ ds,\]
\[\lambda(t)=x_{0}^\mathcal{U} + \int_{0}^t \lambda' (s)\ ds.\]

We claim that
$\gamma' (t)=(\gamma_{n}' (t))^\mathcal{U}$ and  $\lambda' (t)=(\lambda_{n}' (t))^\mathcal{U}$ for a.e. $t$.
Indeed, for $t$ and $h>0$ with $t\pm h \in [0,1]$ we have 
\[\frac{\gamma(t+h)- \gamma(t-h)}{h}=\left(\frac{\gamma_n (t+h)- \gamma_n (t-h)}{h}\right)^\mathcal{U}\]
where
\[\left\|\frac{\gamma_n (t+h)- \gamma_n (t-h)}{h} -2\gamma_{n}' (t)\right\|\leq \beta(h)\]
holds for a.e. $t$ and for \emph{any} $n$ according to \eqref{eq: unif}. Consequently,
\[\left\|\left(\frac{\gamma_n (t+h)- \gamma_n (t-h)}{h}\right)^\mathcal{U} - (\gamma_{n}' (t))^\mathcal{U}\right\|\leq \beta(h)\to 0\quad 
\mathrm{for\ a.e.}\ t\in [0,1], h\searrow 0 .\]
Similarly we see the case with $\lambda$.

Next, observe that 
\[\|\gamma' \| , \ \|\lambda'\|  \in [1/C , C]\]
holds a.e., since the vectors $\gamma_{n}'$ and $\lambda_{n}'$ satisfy the analogous 
statement for every $n$ separately. 

Note that 
\[\gamma' (t) w^\mathcal{U} (\gamma(t)) = (\gamma_{n}' (t) w(\gamma_{n} (t)))^\mathcal{U},\]
\[\lambda' (t) w^\mathcal{U} (\lambda(t)) = (\lambda_{n}' (t) w(\lambda_{n} (t)))^\mathcal{U}\]
for a.e. $t$. Indeed, here we applied the fact that if $\lim_{\mathcal{U}} a_n =a$, then 
$(a_n x_n )^\mathcal{U}= a (x_n )^\mathcal{U}$.

Thus \eqref{eq: wto} and \eqref{eq: nkjU} yield
\[\gamma' (t) w^\mathcal{U} (\gamma(t)) = \lambda' (t) w^\mathcal{U} (\lambda(t))\]
for a.e. $t$.
Now, the assumption of Lemma \ref{lm: exact} are fulfilled, so that $\gamma=\lambda$.

This can be rephrased as follows:
\[\lim_{\mathcal{U}} \|\gamma_n (t) - \lambda_n (t) \|=0,\quad t\in [0,1].\]
Taking into account  \eqref{eq: nkjU}, this means that 
\[\liminf_{k\to\infty} \|\gamma_{n_{k}} (t) - \lambda_{n_{k}} (t))\|=0,\quad t\in [0,1].\]
Since the selection of $(n_k)$ was arbitrary, this yields 
\[\lim_{n\to\infty} \|\gamma_{n} (t) - \lambda_{n} (t))\|=0,\quad t\in [0,1],\]
which completes the proof.
\end{proof}

\section{Quasihyperbolic metric as a renorming technique}\label{sect: renorm}

\begin{prop}\label{prop: equiv}
Let $\Omega \subset \X$ be domain which is a convex, bounded and symmetric (invariant under multiplication by $-1$). Then any quasihyperbolic ball $B=\B_k (0,r)$, $r>0$, defines an equivalent norm in $\X$ by the Minkowski functional
\[M(x) :=\inf\{\lambda> 0 \colon x \in \lambda B\} . \]
\end{prop}
\begin{proof}
It is easy to see that under the assumptions each centered open quasihyperbolic ball $B$ is a norm-open, bounded and symmetric subset. Moreover, it follows from the considerations in \cite{RasilaTalponen12} that $B$ is convex. It is well-known that in such a case the Minkowski functional $M(x)$ defines an equivalent norm in $X$. 
\end{proof}

The smoothness properties of this equivalent norm encode essentially all the relevant smoothness information involving the quasihyperbolic ball. Thus, one may also apply the well-matured Banach space theory to approach the problem of smoothness of the balls. Note that the above norm is different for different radii of the quasihyperbolic ball.

One of the problems in the theory of renormings of Banach spaces (see \cite{DGZ}) is the approximation of the norm of the space by other norms uniformly on bounded sets such that the approximating norm satisfy some nice properties. This task, of course, is sensible only if the space admits an equivalent norm with the given properties in the first place.
 
In most cases the non-separability of the Banach space somewhat complicates the constructions of the nice equivalent norms and in some 
cases it is not clear how the constructions extend from the separable to the non-separable setting (see e.g. \cite{HT}).
 
Above we showed how to induce a norm by a Minkowski functional using a centered quasihyperbolic balls on 
a convex, bounded and symmetric domain. Thus the domain $\Omega$ is essentially an open unit ball given by an equivalent norm $|||\cdot|||$, i.e. $\Omega = \U_{|||\cdot|||}$. 
However, it is easy to see that if we let the quasihyperbolic radius tend to infinity, then the corresponding quasihyperbolic sphere $\partial \B_k (0,r)$ converges to the sphere $\S_{|||\cdot|||}$ in the sense of symmetric Hausdorff distance (taken with respect to any of the equivalent metrics), see the proof of Theorem \ref{thm: UF}. 
This means that the norms $\|x\|_{(r)} := M_{\B_k (0,r)}(x)$ induced by the quasihyperbolic balls approximate $|||\cdot|||$ uniformly on bounded sets. 

This approach has two convenient features. First, as we have seen here, the approximation part comes for free, only thing that is left to do is to analyze the quasihyperbolic balls regarding the favorable property. Secondly, the analysis of the quasihyperbolic balls is usually not sensitive to whether or not the space is separable.
Also, the way the quasihyperbolic inner metric is defined appears to have a  mildly smoothening effect on the resulting norms, per se.

\begin{thm}
Suppose that $(\X,\|\cdot\|)$  is a uniformly smooth Banach space. Then  each equivalent norm can be approximated uniformly on bounded sets by uniformly smooth norms induced by the quasihyperbolic balls, 
similarly as in Proposition \ref{prop: equiv}.

Moreover, if the modulus of smoothness of the norm $\|\cdot\|$ 
has power type $1 < p \leq 2$, then the approximating norms have all the power types less than $(p+1)/2$. 
\end{thm}
\begin{proof}[Sketch of proof] Let $B_n = \B_k (0,n)$. By the proof of Theorem \ref{thm: UF} we know that $k$ is continuously Fr\'echet differentiable on domains $(\partial  B_n)^\eps$ with suitable $\eps>0$. 

Now, since $\partial B_n$ is the level set of a continuously Fr\'echet differentiable function, 
a standard argument employing the Implicit Function Theorem gives that the norm \mbox{$|||\cdot |||_n$}$:=M_{B_n} (\cdot)$ is also continuously Fr\'echet differentiable, see 
\cite[pp. 163--168]{DGZ}.   

In fact, in estimating from above the modulus of smoothness of the $M_{B_n} (\cdot)$ norms, we may restrict by geometric considerations 
to a version of the modulus as follows:
\[\mu_n (\tau)=\sup\left\{\frac{|||y+h|||_n +|||y-h|||_n }{2} -1\colon |||y+h|||_n , |||y-h|||_n \geq |||y|||_n =1,\ \|h\|=\tau\right\}.\]
Here 
$|||y\pm h||| \leq 1 + C d_{\|\cdot\|}(y\pm h,\partial B_n)$ where $C$ is the isomorphism constant involving the equivalent norms $\|\cdot\|$ and
$|||\cdot|||$. This constant depends on the particular norm $|||\cdot|||_n =M_{B_n} (\cdot)$ but it may be replaced by one universal to all the norms $M_{B_n} (\cdot)$, $n\in \N$.

Suppose that $|||y\pm h||| \geq 1$. Then 
\[k_0 (y\pm h) \geq k_0 (y) + d_{\|\cdot\|} (y\pm h , \partial B_n ) \inf \{1/d(x , \partial \Omega) \colon x\in  \partial B_n\} \]
by the definition of the $k$ metric. Thus, for large $n$ the above infimum becomes larger than $C$, so that asymptotically
\begin{multline*}
\sup\left\{\frac{|||y+h|||_n +|||y-h|||_n }{2} -1\colon |||y+h|||_n , |||y-h|||_n \geq |||y|||_n =1,\ \|h\|=\tau\right\}\\
\leq \sup\left\{\frac{k_0 (y+h) +k_0 (y-h)}{2} -k_0 (y)\colon k_0 (y+h) , k_0 (y-h) \geq k_0 (y),\ \|h\|=\tau\right\}
\end{multline*}
holds. The latter quantity is controlled in the proof of Theorem \ref{thm: UF}.
\end{proof}

Recall that the modulus of convexity $\delta$ need not be a convex function. However, it has a greatest convex minorant and we denote this 
by $\hat{\delta}$.

\begin{thm}\label{thm: unifconv}
Suppose that $\X$ is a uniformly convex Banach space, $\Omega \subset \X$ is a symmetric convex domain and 
$B=B_k (0,r)$ for some $r>0$. Then the equivalent norm $|||\cdot |||= M_B (\cdot)$ on $\X$ is uniformly convex.

Moreover, any equivalent norm on $\X$ can be approximated uniformly on bounded sets by uniformly convex norms 
arising from such quasihyperbolic balls $B$.
\end{thm}
\begin{proof}
Fix $R>0$. Let $B=\B_k (0,R)$ and $|||\cdot|||=M_B (\cdot)$.
Note that the formal identity mappings $(B,k) \to (B, \|\cdot\|)$ and $(B,k) \to (B, |||\cdot |||)$ are bilipschitz due to the fact that 
on $B$ the weight $\frac{1}{d(x,\partial \Omega)}$ is bounded, see the proof of Theorem \ref{thm: UF}.

Let $x_n ,y_n \in B$ with $k(0,x_n )=k(0,y_n)=R$ and $k(0,(x_n + y_n)/2)\to R$. In order to prove the uniform convexity of 
$|||\cdot|||$ it suffices to prove that $\|x_n -y_n\|\to 0$ by the equivalence of the norms.
There exists by the uniform convexity of $\X$ and the convexity of the domain 
unique geodesics $\gamma_n$ and $\lambda_n$ from $0$ to $x_n$ and $y_n$, respectively.

Without loss of generality we may assume that $\gamma_n$ and $\lambda_n$ are parametrized by the quasihyperbolic length. 
Then
\[k(0,x_n)=\int_{0}^{R} \frac{\|\gamma_{n}'(t)\|}{d(\gamma_n (t), \partial \Omega)}\ dt = k(0,y_n )=\int_{0}^{R} \frac{\|\lambda_{n}'(t)\|}{d(\lambda_n (t), \partial \Omega)}\ dt =R\]
where
\[v_n: (t) = \frac{\gamma_{n}'(t) }{d(\gamma_n (t), \partial \Omega)}\ \mathrm{and}\ w_n (t) := \frac{\lambda_{n}'(t)}{d(\lambda_n (t), \partial \Omega)} \] 
are norm-$1$ for a.e. $t\in [0,R ]$. Also,
\[\|x_n -y_n \|=\|\gamma_n (R) - \lambda_n (R)\| \leq \int_{0}^R \|(\gamma_n - \lambda_n )' \|\ dt.\]

Note that 
\[k\left(0,\frac{x_n + y_n }{2}\right) \leq \int_{0}^{R} \frac{\frac{1}{2}\|(\gamma_n (t)+\lambda_n (t))'\|}{ d(\frac{1}{2}(\gamma_n (t) +\lambda_n (t)) ,\partial\Omega)}\ dt \leq \int_{0}^R \frac{\|v_n (t) + w_n (t)\|}{2}\ dt, \]
where we use the convexity of the domain, see Lemma \ref{lm: PLconvex}. 
Put 
\[\mathrm{Aver}_n =\frac{1}{R}\int_{0}^{R} \|v_n (t) - w_n (t)\|\ dt .\]

Observe that
\begin{multline*}
 \int_{0}^{R} \frac{\|v_n (t)+ w_n (t)\|}{2}\ dt   \leq  \int_{0}^{R} 1- \hat{\delta}_{\|\cdot\|}(\|v_n (t)- w_n (t)\|)\ dt \\
 \leq  \int_{0}^{R} 1- \hat{\delta}_{\|\cdot\|}(\mathrm{Aver}_n )\ dt =  R - R\hat{\delta}_{\|\cdot\|}(\mathrm{Aver}_n ) .
\end{multline*}

The conclusion from this is that if $k(0,(x_n + y_n)/2)\to R$ as $n\to\infty$, then $\mathrm{Aver}_n \to 0$ as $n\to\infty$.
We will complete the proof by applying the variational principle of Theorem \ref{thm: UF}.
Then it follows that $x_n = \gamma_n (R)$, $y_n = \lambda_n (R)$ satisfy $\|x_n -y_n \|\to 0$ which then implies the 
uniform convexity of  $|||\cdot |||$.

Finally, we will check that the assumptions of Theorem \ref{thm: UF} hold.
Since $\X$ is uniformly convex it is in particular superreflexive.
Observe that the weight $w(x)=1/d(x,\partial \Omega)$ is Lipschitz and bounded below 
by
\[\frac{1}{\sup_{x\in B} d(x,\partial\Omega)}>0\text{ in }B. 
\]

Since the paths are parametrized by the quasihyperbolic length, we can choose $C$ appearing in the assumptions to be 
$C=\max\{\sup_{x\in B} d(x,\partial\Omega),1/\inf_{x\in B} d(x,\partial\Omega)\}$.
Since $\mathrm{Aver}_n \to 0$ we get the `if' part of the Theorem.

To obtain \eqref{eq: unif} we apply the proof of Theorem 3.1 in \cite{RasilaTalponen14} with a modification. 
The condition \eqref{eq: unif} is obtained by using the triangle inequality 
on \cite[$(3.2)$]{RasilaTalponen14} (sic). 

Here $\X$ is a uniformly convex space and $w$ is Lipschitz continuous with constant 
\[
\sup_{x\in B} \frac{d}{dt} \frac{1}{t}\bigg|_{t=d(x,\partial \Omega)}.
\] 
We will relax the assumption about the power type of the modulus of convexity as follows.
The largest convex minorant $\hat{\delta}_\X$ of $\delta_\X$ is strictly increasing (see e.g. \cite[p. 85]{SC}), 
therefore $\hat{\delta}_{\X}^{-1}$ exists. 
One can use the definition $\beta(h)=\hat{\delta}_{\X}^{-1}(\mu(h))$ to get \cite[$(3.2)$]{RasilaTalponen14}. This convention 
will not be sufficient to obtain \cite[$(3.3)$]{RasilaTalponen14}, a priori, but that is not required for our purposes.

This concludes the proof.
\end{proof}

\begin{prop}
Suppose that $\X$ is a strictly convex Banach space with the RNP. Then each equivalent norm can be approximated uniformly on bounded sets 
by strictly convex norms induced by the quasihyperbolic balls, similarly as above.
\end{prop}
\begin{proof}
The argument is based on the fact that quasihyperbolic balls in a convex domain of a strictly convex Banach space with the RNP are strictly convex as well,
see \cite{RasilaTalponen12}. Note that it is not essential here whether $\overline{\Omega}$ is strictly convex or not.
\end{proof}

We do not know if quasihyperbolic balls in symmetric convex domains of reflexive LUR Banach spaces and centered at the origin induce LUR norms 
via the Minkowski functional.

\section{Non-geodesic domains}

There exists examples of domains in Hilbert spaces, which are not geodesic with respect to the quasihyperbolic metric. For more details see \cite[Example 2.9]{Vai98} and \cite[Remark 3.5]{Vai99}. However, in these examples the complement of the domain is uncountable. We give an example of a domain that is not geodesic with respect to the quasihyperbolic metric and the complement of the domain is countable.

\begin{example}
  Let
  \[
    \Omega = \ell^2 \setminus \left( \{ 0 \} \cup \left\{ \pm \sqrt{2}\left( 1-\frac{1}{i} \right) e_i \right\}_{i=2}^\infty \right).
  \]
  There does not exists a quasihyperbolic geodesic from $-e_1$ to $e_1$ in $\Omega$.
\end{example}
\begin{proof}
  In $\Omega' = \ell^2 \setminus \{ 0 \}$ the geodesics $-e_1 \curvearrowright e_1$ are half circles with center at 0. Since $G' \subset G$ we have
  \[
    k_\Omega(-e_1,e_1) \ge k_{\Omega'}(-e_1,e_1) = \pi.
  \]
  Let $\gamma$ be a curve joining $-e_1$ to $e_1$. If $\gamma$ is not a half circle, then
  \[
    k_\Omega(\gamma) \ge k_{\Omega'}(\gamma) > \pi.
  \]
  If $\gamma$ is a half circle, then for some $n$ we have $\min \{ d(e_n,\gamma) , d(-e_n,\gamma) \} < 1$ implying $k_\Omega(-\gamma) > \pi$.
  
  Let us consider $\gamma_n$ to be a half circle with center at origin, end points $-e_1$ and $e_1$ such that for each $x = (x_1,x_2,\dots) \in \gamma_n$ we have
  \[
    x_1 \in [-1,1], \quad 0 = x_2 = \cdots = x_{n-1} = x_{n+2} = \cdots, \quad x_{n} = x_{n+1}
  \]
  for some $n \ge 2$. By construction of $\Omega$ we have $\ell_k (\gamma_n) > \ell_k (\gamma_{n+1}) > \pi$ and $\lim_{n \to \infty} \ell_k (\gamma_n) = \pi$. Thus, there does not exist a quasihyperbolic geodesic.
\end{proof}

\section{On uniqueness of geodesics in finite-dimensional domains}

Finally, we take this opportunity to address some related finite-dimensional questions involving the quasihyperbolic metric that arouse 
during our research.

In the upper half-plane the quasihyperbolic geodesics agree with the hyperbolic geodesics and are thus unique. However, if a plane domain is not simply connected the quasihyperbolic geodesics need not be unique. The following results shows that for any $n \ge 3$ there exists a domain and at least two geodesics with exactly $n$ common points.

\begin{lem}\label{lem:geodesic result}
  For each $n=2,3,\dots$ there exists a domain $G \subsetneq \R^2$, points $x,y \in \Omega$ and quasihyperbolic geodesics $\gamma_1 \colon x \curvearrowright y$ and $\gamma_2 \colon x \curvearrowright y$ such that $\# ( \gamma_1 \cap \gamma_2 ) = n$.
\end{lem}
\begin{proof}
  In the case of $n=2$ we can choose $\Omega = \R^2 \setminus \{ 0 \}$, $x = -e_1$ and $y = e_1$. Now there exists exactly two geodesics joining $x$ and $y$, which are half circles, and we choose them to be $\gamma_1$
 and $\gamma_2$. We obtain $\# ( \gamma_1 \cap \gamma_2 ) = 2$ and the assertion follows.
 
  We assume that $n \ge 3$. Let us first construct the domain $\Omega$. We define rectangular $R_n = \{ z = (z_1,z_2) \in \R^2 \colon z_1 \in (-1,(n-2)\sqrt{3}+1) , \, z_2 \in (-1,1) \}$ and point set $P_n = \{ 0, \sqrt{3}, \dots , (n-2)\sqrt{3} \}$. Next we define the set
  \[
    L_n = \left\{ z \in \R^2 \colon z_2 \le -\tfrac{1}{2}, \, z_2 \ge \tfrac{1}{2} , \, z_1 \in \left\{ \tfrac{\sqrt 3}{2}, \sqrt{3}+\tfrac{\sqrt 3}{2}, \dots , (n-3)\sqrt{3} + \tfrac{\sqrt 3}{2} \right\} \right\}.
  \]
  We define
  \[
    \Omega_n = R_n \setminus \left( P_n \cup L_n \right).
  \]
  \begin{figure}[ht]
    \includegraphics[width=10cm]{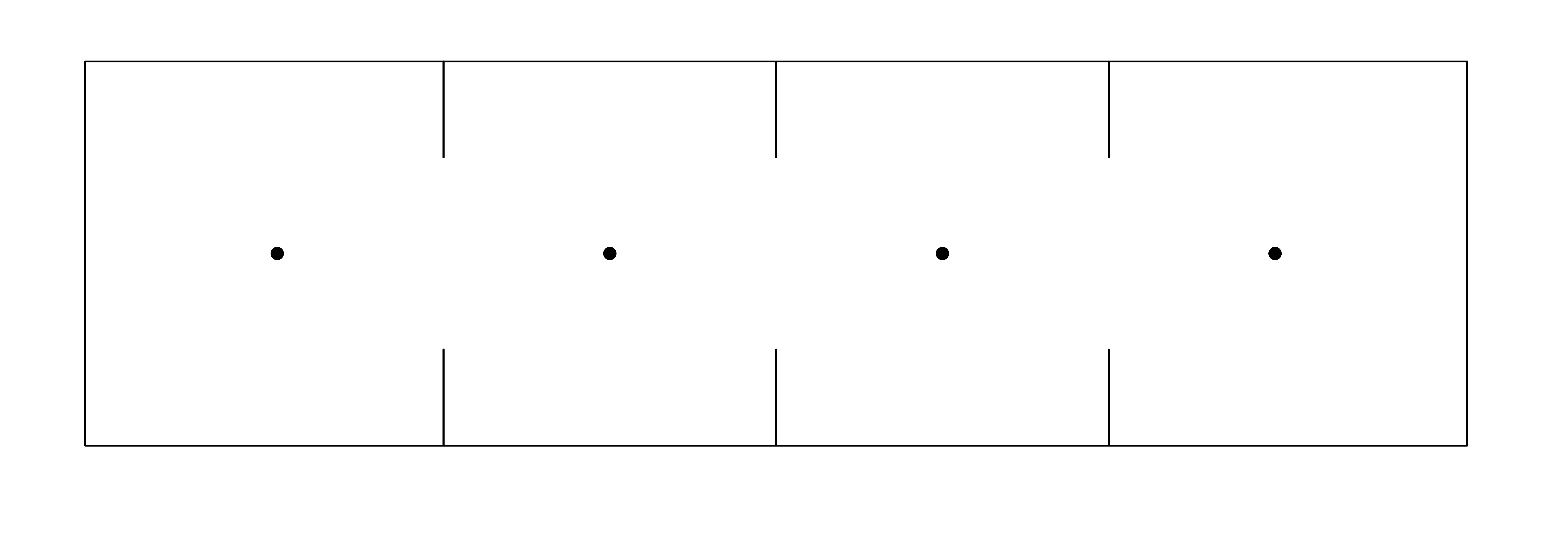}
    \caption{The domain $\Omega_5 = R_5 \setminus \left( P_5 \cup L_5 \right)$.} \label{fig:example domain}
  \end{figure}

  We choose $x = -\tfrac{1}{2}e_1$ and $y = ((n-2) \sqrt{3} + \tfrac{1}{2}) e_1$. Finally, we find the geodesics $\gamma_1$ and $\gamma_2$. Let $\gamma_1$ be the geodesic, which consists of circular arcs and is contained in the closed upper half-plane $\{ z = (z_1,z_2) \in \R^2 \colon z_1 \ge 0 , \, z_2 \in \R_n \}$, see Figure \ref{fig:example domain}. Then
  \[
    \gamma_1 \cap \{ z = (z_1,z_2) \colon z_1 = 0 \} = \left\{ 0, \tfrac{\sqrt{3}}{2}e_1, \tfrac{3\sqrt{3}}{2}e_1, \dots , \tfrac{(2n-5)\sqrt{3}}{2}e_1, \left( (n-2)\sqrt{3} + \frac12 \right) e_1 \right\}.
  \]
  \begin{figure}[ht]
    \includegraphics[width=75mm]{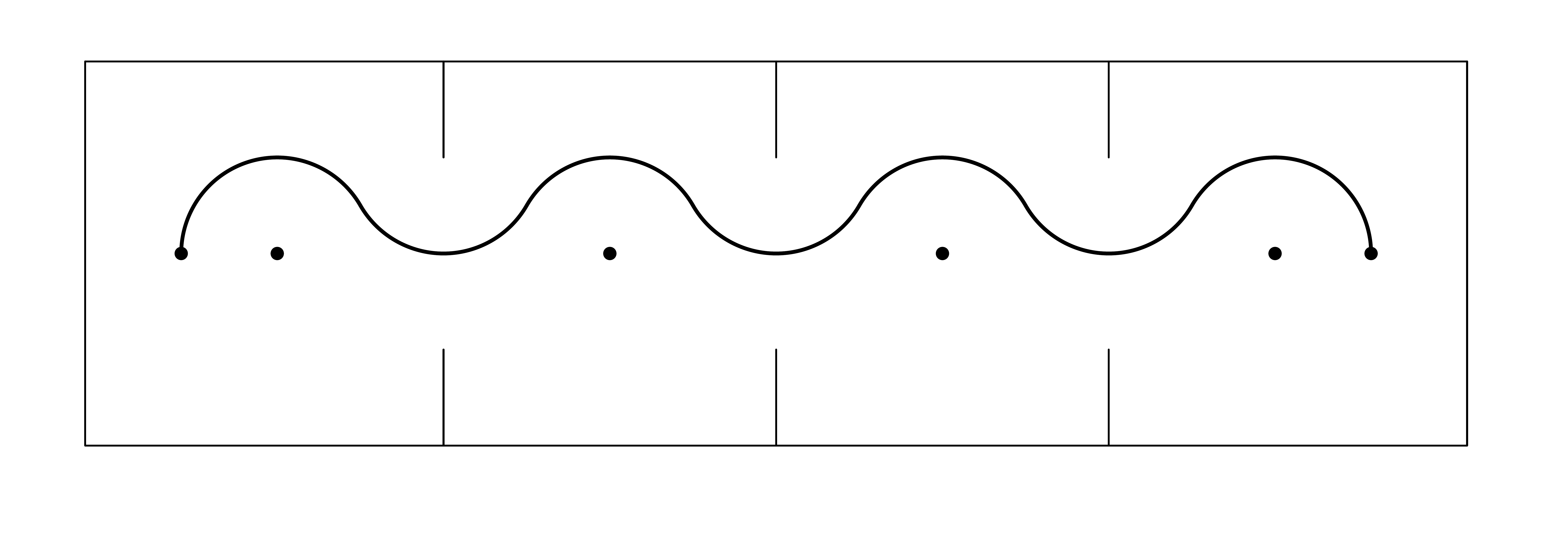}\hspace{5mm}
    \includegraphics[width=75mm]{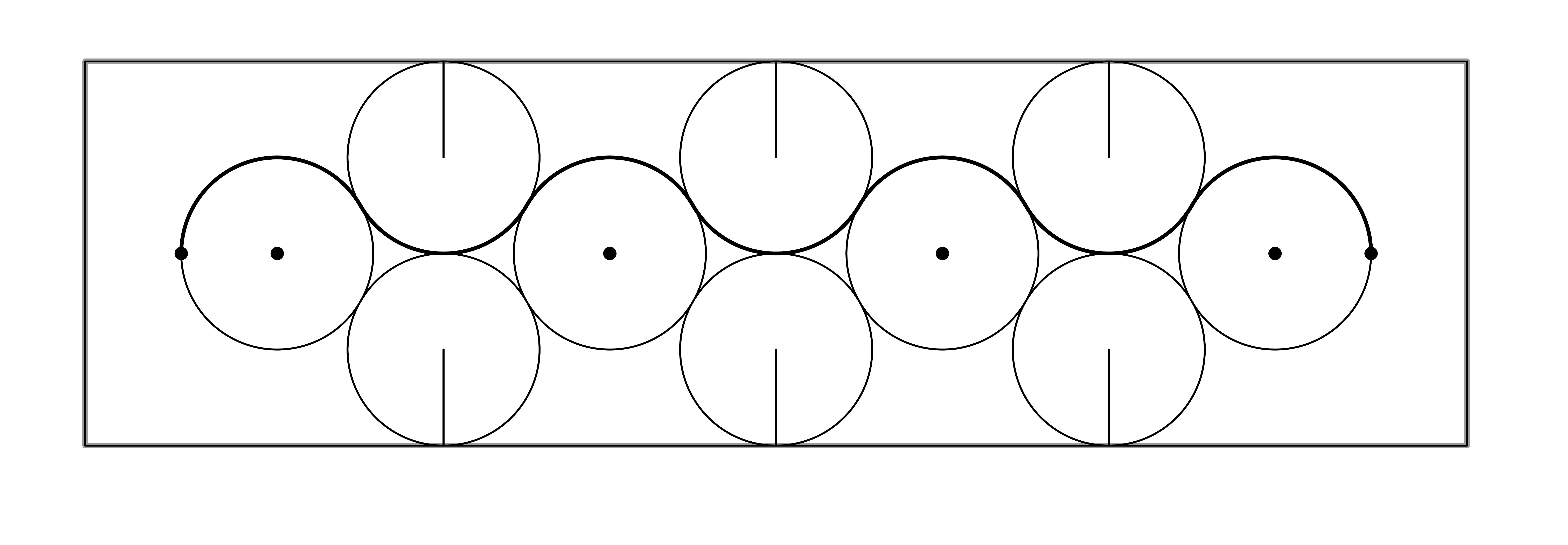}
    \caption{Geodesic $\gamma_1$ in the proof of Lemma \ref{lem:geodesic result}.} \label{fig:example domain}
  \end{figure}
   Let $\gamma_2$ be the reflection of $\gamma_1$ across the $x_1$-axis. Now
  \[
    \gamma_1 \cap \gamma_2 = \left\{ 0, \tfrac{\sqrt{3}}{2}e_1, \tfrac{3\sqrt{3}}{2}e_1, \dots , \tfrac{(2n-5)\sqrt{3}}{2}e_1, \left( (n-2)\sqrt{3} + \frac12 \right) e_1 \right\}
  \]
  implying $\# ( \gamma_1 \cap \gamma_2 ) = n$ and the assertion follows.
\end{proof}

\begin{rem}
  Note that Lemma \ref{lem:geodesic result} is true for all geodesics $\gamma_1$ and $\gamma_2$ joining the points $x$ and $y$. There are $2^{n-1}$ of such geodesics.
\end{rem}

Next we give an example of a strictly starlike domain, which contains arbitrarily short quasihyperbolic geodesics, which cannot be uniquely prolonged.

\begin{example}\label{example:prolongation}
  Let us consider the polygon $P \subset \R^2$ with vertices at $(-4,1)$, $(-1,1)$, $(-1,4)$, $(4,4)$, $(4,-4)$, $(-1,-4)$, $(-1,-1)$ and $(-4,-1)$.
  
  Let first $x=(-2,0)$, $y=(-1,0)$, $z_1=(0,1)$ and $z_2=(0,-1)$. Now the quasihyperbolic geodesic $x \curvearrowright y$ is the Euclidean line segment $[x,y]$ and geodesics $x \curvearrowright z_1$ and $x \curvearrowright z_2$ pass through point $y$ and contain the geodesic $x \curvearrowright y$. Thus the geodesic $x \curvearrowright y$ cannot be uniquely prolonged.
  
  \begin{figure}[ht]
    \includegraphics[height=50mm]{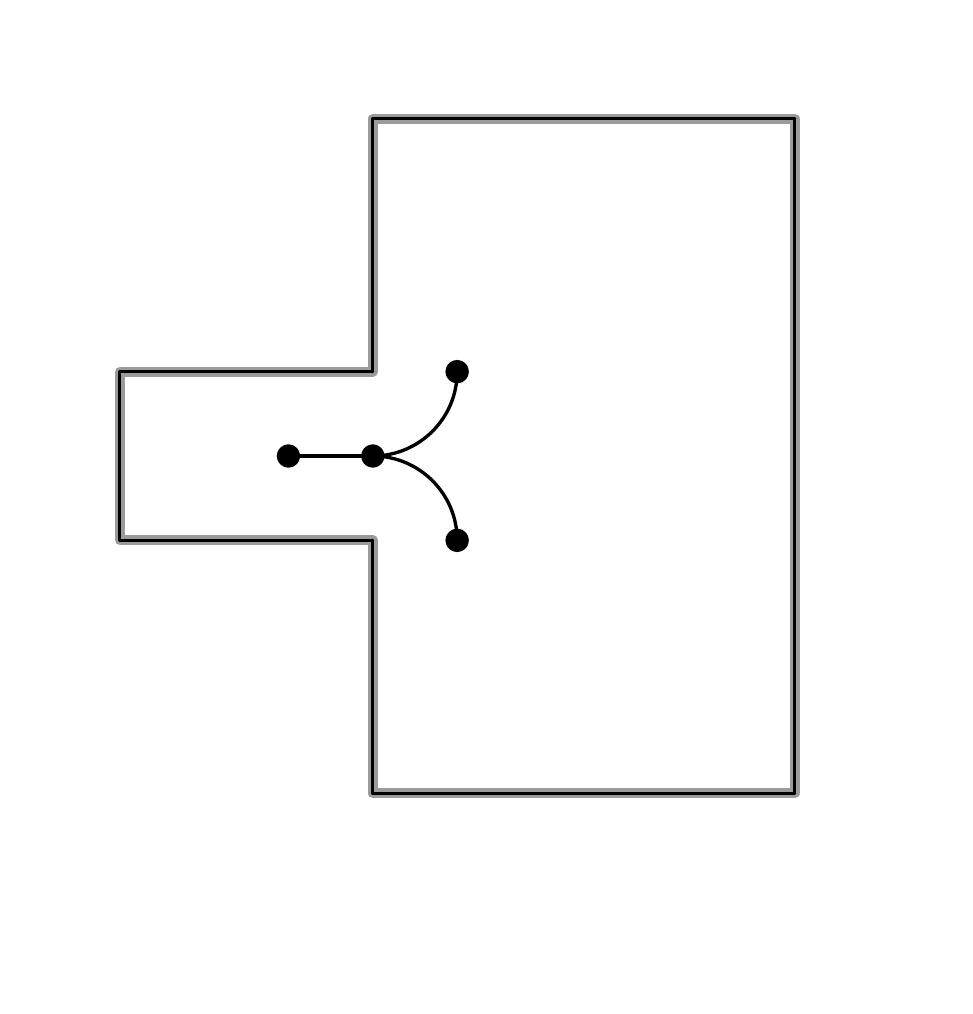}
    \caption{Polygon $P$ and geodesics in Example \ref{example:prolongation}.}
  \end{figure}  
  
  Let then $x=(-t-1,0)$ for some $t \in (0,1)$. By simple computation we notice that $k_P(x,y) = t$ and similarly as above the geodesic $x \curvearrowright y$ cannot be uniquely prolonged.
\end{example}

\begin{prop}
There exists a star-like domain $\Omega \subset \R^n$, $n\geq 3$, such that not all the quasihyperbolic geodesics are unique. Moreover, this domain 
can be chosen in such a way that the function $x \mapsto 1/d(x,\partial \Omega)$ does not have any local maxima. 
\end{prop}
\begin{proof}
Consider the following domain
%
%
\[
  \Omega = \left\{(x_i)\in \R^n \colon 0 < \sqrt{\sum_{i<n} |x_i|^2 } < 1 , \, x_n\ge \frac{1}{2} \right\} \cup \left\{(x_i)\in B(1/2,1) \colon  x_n < 1/2 \right\}.
\]
Now the geodesics from $x=x_1/2 + x_n$ to $y=-x_1/2+x_n$ are not unique. Clearly, the function $x \mapsto 1/d(x,\partial \Omega)$ does not have any local maxima, and its global maximum is $1/2$.
\end{proof}

\end{document}